\documentclass[11pt,a4paper,reqno]{amsart}

\usepackage{expl3}
\expandafter\def\csname ver@l3regex.sty\endcsname{} 

\usepackage{amsfonts}
\usepackage{graphicx}
\usepackage{graphics}
\usepackage{epsfig}
\usepackage{anysize}
\usepackage{lipsum}
\usepackage{wrapfig}
\usepackage{coloring}

\marginsize{3cm}{3cm}{3cm}{3cm}
\newtheorem{theorem}{Theorem}[section]

\newtheorem{lemma}{Lemma}[section]

\newtheorem{definition}{Definition}[section]

\makeatletter

\begin{document}

\setcounter{page}{1}
\vspace{2cm}
%----------------------------------------------------------------------------------------------------------------
\author[\hspace{0.7cm}\centerline{TWMS J. App. Eng.
Math. V.xx, N.xx, 20xx}]{S.S. B\.ILG\.IC\.I$^1$, M. \c{S}AN$^2$,\S }
\title[\centerline{S.S. B\.ILG\.IC\.I, M. \c{S}AN: Existence and Uniqueness Results ...
\hspace{0.5cm}}]{Existence and Uniqueness Results 
for a  
nonlinear fractional differential 
equations \\
of order $\sigma\in(1,2)$}

\thanks{\noindent $^1$ Çankırı Karatekin University, Faculty of Sicence, Uluyazı Campus, 18100, Çankırı, Turkey.\\
\indent \,\,\, e-mail: sinan.serkan.bilgici@gmail.com; ORCID: https://orcid.org/0000-0002-4237-6204.\\
\indent $^2$ Çankırı Karatekin University, Faculty of Sicence, Uluyazı Campus, 18100, Çankırı, Turkey.\\
\indent \,\,\, e-mail: mufitsan@karatekin.edu.tr; ORCID: https://orcid.org/0000-0001-6852-1919.\\
\indent \S \, Manuscript received: Month Day, Year; accepted: Month Day, Year. \\
\indent \,\,\,.\\
}

\begin{abstract}
The main objective of this article is to discuss the local existence of the solution to an initial value problem involving a non-linear differential equation in the sense of Riemann-Liouville fractional derivative of order $\sigma\in(1,2),$ when the nonlinear term has a discontinuity at zero. Hereafter, by using some tools of Lebesgue spaces such as Hölder inequality, we  obtain Nagumo-type, Krasnoselskii-Krein-type and Osgood-type uniqueness theorems for the problem. 

 
\end{abstract}
\maketitle 
\bigskip

%
%%%%%%%%%%%%%%%%%%%%%%%%%%%%%%%%%%%%%%%%%%%%%%%%%%%%%%%%%%%%%%%%%%%%%%%%%%%%%%%%%%%%%
\section{Introduction}
The importance of existence and uniqueness theorems for initial value and boundary value problems (IVP and BVP) involving the classical derivative operator  is indisputable because, without them, one cannot understand modelled systems correctly and make predictions how they behave. Recently, with the popularity of fractional derivative operators such as Riemann-Liouville (R-L), Caputo (C), etc., the equations involving these operators have begun to be studied in detail (See, \cite{Del}-\cite{Yoruk}). However, such a generalization leads to some difficulties and differences. For instance, unlike the initial value problems involving the classical derivative, the existence of continuous solution to some IVPs in the sense of  
R-L derivative strictly depends on the initial values and smoothness conditions on the functions in right-hand side of equations in IVPs. To support this claim, one can refer to \cite{San3}, and one can see there that an initial value problem including a non-linear fractional differential equation of order $\sigma\in (0,1)$ has no continuous solution when the problem has a non-homogeneous initial value and the right-hand side of the equation is continuous on $[0,T]\times\mathbb{R}$. The similar issue arises in our  investigation of the existence and uniqueness of solutions to the following problem
\begin{equation} \label{intvalue}
\begin{cases}
&D^{\sigma}\omega(x) = f\big(x,\omega(x),D^{\sigma-1}\omega(x)\big),\quad x>0   \\  
&\omega(0)=0, \quad D^{\sigma-1}\omega\left(x\right)|_{x=0} =b,    
\end{cases}
\end{equation} 
where $\sigma\in(1,2),$ $b\neq 0,$ $f$ will be specified later and $D^\sigma$ represents the Riemann-Liouville fractional derivative of order $\sigma,$ which is given by
$$D^{\sigma}\omega(x) =\frac{1}{\Gamma(2-\sigma)}\frac{d^{2}}{dx^{2}}\int_{0}^{x}\frac{\omega(t)}{(x-t)^{\sigma-1}}dt.
$$

The equation in \eqref{intvalue} was first considered by Yoruk et. al. \cite{Yoruk}, when the second initial value is also homogenous ($b=0$) and the right-hand side function is continuous on $[0,T]\times \mathbb{R}\times\mathbb{R}.$ They gave Krasnoselskii-Krein, Roger and Kooi-type uniqueness results. Since problem \eqref{intvalue} we consider has a non-homogeneous initial condition, we investigate existence of its solutions in a convenient space of functions under following conditions:

\begin{itemize}
	\item[(C1)]
	\quad Let $f\left(x,t_{1},t_{2}\right)\in C \big(\left(0,T\right]\times\mathbb{R}\times\mathbb{R}\big)$ and $x^{\sigma-1} f\left(x,t_{1},t_{2}\right) \in C\big(\left[0,T\right]\times\mathbb{R}\times\mathbb{R}\big),$ 
\end{itemize}
where $C(X)$ represents the class of continuous functions defined on $X.$  

Moreover, under some appropriate conditions we establish Nagumo-type, Krasnoselskii-Krein-type and Osgood-type uniqueness results for the problem. To prove them, we follow some ways introduced in \cite{Agarwal},\cite{San3},\cite{Yoruk} by generalizing some definitions made there and, in addition to this, we use the tools of Lebesgue spaces such as Hölder inequality. 

\section{Preliminaries} 

We begin with definitions of R-L integral and R-L derivative of higher order. The lower terminal points of integrals in their formulas will be taken as zero.   

\begin{definition}
	
	\label{def1} The Riemann-Liouville integral of order $\sigma>0$ of a function $\omega\left(x\right)$ is defined by
	\begin{equation}
	I^{\sigma}\omega\left(x\right) :=\frac{1}{\Gamma \left(\sigma\right) }%
	\ \int_{0}^{x}\frac{\omega\left( t\right) }{\left( x-t \right)
		^{1-\sigma}}dt 
	\end{equation}
	provided that the integral is pointwise defined on $\mathbb{R}_{0}^{+}.$
\end{definition}

\begin{definition}
	The Riemann-Liouville derivative of order $n-1<\sigma<n$ ($n\in\mathbb{N}$) of a function $\omega\left(x\right)$ is given by
	\begin{equation} \label{def2} 
	D^{\sigma}\omega\left(x\right) =\frac{1}{\Gamma \left( n-\sigma\right) 
	}\frac{d^{n}}{dx^{n}}\int_{0}^{x}\frac{\omega\left(t \right) }{\left(x-t
		\right) ^{\sigma-n+1}}dt,  v
	\end{equation}
	provided that the right side is pointwise defined on $\mathbb{R}_{0}^{+}.$
\end{definition}  

Compositional relations are frequently used in the literature (See, \cite{Kilbas}, \cite{Podlubny}) for converting an initial value or boundary-value problem to a corresponding integral equation. The following two lemmas \cite{Podlubny} are related to that.  

\begin{lemma}\label{lemma}

	For a function $\omega(x)$ such that $D^{\sigma}\omega(x)$ with  $n-1<\sigma<n$ is integrable, the compositional relation    
	\begin{equation} 
	I^{\sigma}D^{\sigma}\omega(x)=\omega(x)+\sum_{m=1}^{n}D^{\sigma-n}\omega(x)|_{x=0} \frac{x^{\sigma-n}}{\Gamma(\sigma-n+1)}
	\end{equation}
	is satisfied, where it is assumed that $D^{\sigma-n}\omega(x)=I^{n-\sigma}\omega(x)$ when $\sigma<n.$

	In case of $\sigma=2,$ the above formula turns into
	\begin{equation}
	I^{\sigma}D^{\sigma}\omega(x)=\omega(x)+D^{\sigma-1}\omega(x)|_{x=0} \frac{x^{\sigma-1}}{\Gamma(\sigma)}+D^{\sigma-2}\omega(x)|_{x=0} \frac{x^{\sigma-2}}{\Gamma(\sigma-1)}. 
	\end{equation}
	
	Moreover, if $\omega$ is continuous on $[0,T],$ then  
	\begin{equation}
	I^{\sigma}D^{\sigma}\omega(x)=\omega(x)+D^{\sigma-1}\omega(0)\frac{x^{\sigma-1}}{\Gamma(\sigma)}
	\end{equation}	 
	holds.  	 
\end{lemma}  

\begin{lemma}\label{lemma1.1}   
	For an $\sigma$ R-L integrable function $\omega(x),$ the well-known rule 
	$$D^{\sigma}I^{\sigma}\omega(x)=D^{2}I^{2-\sigma}I^{\sigma}\omega(x)=D^{2}I^{2}\omega(x)=\omega(x)$$
	holds. 
\end{lemma}   

Solutions to problem \eqref{intvalue} will be investigated in the space defined below \cite{Del} :  

\begin{theorem}\label{Th1}  The space of continuous functions defined on $[0,T],$ whose R-L fractional derivative of order $\beta,$ $0<\beta<1$ are continuous on  $[0,T]$ is a Banach space when endowed with the following norm:
	\begin{align*}
	||\omega||_{\beta}=||\omega||_{\infty}+||D^{\beta}\omega||_{\infty},
	\end{align*}
	where $||.||_{\infty}$ is the supremum norm defined on the class of continuous functions. 	This space will be denoted by $C^{\beta}([0,T]).$ 
\end{theorem} 

The local existence of solutions  to problem \eqref{intvalue} will be proved with the aid of Schauder fixed point theorem \cite{Zeidler}:  

\begin{theorem}\label{Th2} Let $\mathcal{C}$ be a closed, bounded, convex subset of a Banach space $X:=\big\{u:I\to\mathbb{R} \ \text{continuous}  : I\subset\mathbb{R} \ \text{closed and bounded interval}\big\}.$  If operator $\mathcal{S}:\mathcal{C}\rightarrow \mathcal{C}$ is continuous and, if $\mathcal{S}(\mathcal{C})$ is an equicontinuous set on $I,$  
	then $\mathcal{S}$ has at least one fixed point in $\mathcal{C}.$  
\end{theorem}  

 \section{Main Results} 

The one of  mathematical tools used for showing the existence and uniqueness of the desired type of solution to a given initial or boundary value problem is first to convert them into an integral equation. One investigates the existence and uniqueness of the solution to the integral equation instead of the associated problem. Here, we follow this way by taking the aid of the lemma given below: 

\begin{lemma} \label{lemma1} Under condition (C1),
	if $\omega \in C^\sigma [0,T]$ is a solution of problem (\ref{intvalue}), then   $\omega\in C^{\sigma-1}[0,T]$ is a solution of the following integral equation 
	\begin{align} 
	\omega(x)&=\frac{b}{\Gamma(\sigma)}x^{\sigma-1}+\frac{1}{\Gamma(\sigma)}\int_{0}^{x}\frac{f(t,\omega(t),D^{\sigma-1}\omega(t))}{(x-t)^{1-\sigma}}dt  \label{intequ}
	\end{align}
	and, vice versa.
\end{lemma}

\begin{proof} We assume that $\omega \in C^\sigma [0,T]$ is a solution of problem (\ref{intvalue}). If we apply $I^{\sigma}$ to both sides of the equation in the problem and, if we consider  
	$$I^{\sigma}D^{\sigma}\omega(x)=\omega(x)+\frac{D^{\sigma-1}\omega(0)}{\Gamma(\sigma)}x^{\sigma-1} \quad \text{for all} \quad  u \in C^\sigma [0,T],$$ 
	then the integral equation in (\ref{intequ}) is appeared.   	
	
	Now we suppose that  $\omega\in C^\sigma [0,T]$ is a solution of  integral equation (\ref{intequ}), and let us show that $u$ is a solution of the problem (\ref{intvalue}).  If  $D^{\sigma}$ is applied to the both sides of (\ref{intequ}), and then, if    
	$$D^{\sigma}I^{\sigma}\omega(x)=\omega(x) \ \ \text{for all} \ \ \omega\in C^\sigma [0,T]$$ 
	is used, then one can observe that $\omega \in C[0,T]$ satisfies the equation in (\ref{intvalue}). Moreover, let us prove that $\omega \in C^\sigma [0,T]$ also fulfils  initial value conditions. By change of variables and condition (C1) we have
	\begin{align}
	\omega(0)=\lim_{x\rightarrow 0^+}\omega(x)&=\frac{b}{\Gamma(\sigma)}x^{\sigma-1}+\frac{1}{\Gamma(\sigma)}\lim_{x\rightarrow0^+}\int_{0}^{x}\frac{f(t,\omega(t),D^{\sigma-1}\omega(t))}{(x-t)^{1-\sigma}}dt\nonumber  \\
	&=\frac{1}{\Gamma(\sigma)}\lim_{x\rightarrow 0^+}\int_{0}^{x}\frac{t^{\sigma-1}f(t,\omega(t),D^{\sigma-1}\omega(t))}{t^{\sigma-1}(x-t)^{1-\sigma}}dt \nonumber \\
	&=\frac{1}{\Gamma(\sigma)}\lim_{x\rightarrow 0^+}x\int_{0}^{1}\frac{(x\tau)^{\sigma-1} f(x\tau,\omega(x\tau),D^{\sigma-1}\omega(x\tau))}{\tau^{\sigma-1} (1-\tau)^{1-\sigma}}d\tau =0,\label{lem2}   
	\end{align}
	showing that $u$ satisfies the first initial condition in (\ref{intvalue}). 
	
	Now let us show that  $u$ provides the second initial condition in (\ref{intvalue}). If $D^{\sigma-1}$ is applied to both sides of (\ref{intequ}), and if the relation  $D^{\sigma-1}I^{\sigma}h(x)=Ih(x)$  is used, then we can first get
	\begin{equation}
	D^{\sigma-1}\omega(x)=D^{\sigma-1}\omega(0)+\int_{0}^{x}f(t,\omega(t),D^{\sigma-1}\omega(t))dt.
	\end{equation} 
	From here, we then obtain
	\begin{align}
	D^{\sigma-1}\omega\left(0\right)&= b+\lim_{x\rightarrow 0^+}\int_{0}^{x}f(t,\omega(t),D^{\sigma-1}\omega(t))dt\nonumber \\
	&=b+\lim_{x\rightarrow 0^+} \int_{0}^{x}\frac{1}{t^{\sigma-1}}t^{\sigma-1}f(t,\omega(t),D^{\sigma-1}\omega(t))dt \nonumber\\
	&=b+\lim_{x\rightarrow 0^+} \int_{0}^{1}\frac{x^{2-\sigma}}{\tau^{\sigma-1}}(x\tau)^{\sigma-1}f(x\tau,\omega(x\tau),D^{\sigma-1}\omega(x\tau))d\tau=b, \label{lem3}
	\end{align}
	since $2-\sigma>0$  and $t^{\sigma-1}f(t,\omega(t),D^{\sigma-1}\omega(t))$ is continuous on $[0,T].$ 
	Consequently, it has been shown that a solution of (\ref{intequ}) provides the problem (1.1) if condition (C1) is assumed to be satisfied. 
\end{proof}

\begin{theorem} [Existence] Let condition (C1) be satisfied, and assume that there exist positive real numbers $r$ and $\mathcal{M}$ such that
	$\left|x^{\sigma-1}f(x,\omega,v)\right|\leq \mathcal{M} \quad \text{for all} \quad (x,\omega,v)\in I=\left[0,T\right]\times\left[-r_{1},r_{1}\right] \times \left[ b-r_{2},b+r_{2}\right]$ 
	with $r_{1}+r_{2}\leq r.$  Then problem (\ref{intvalue}) admits at least one solution in  $C^\sigma [0,T_{0}],$ where 
	\begin{equation}\label{T0}
	T_{0}= \begin{cases}
	\quad T \quad &\text{if} \quad T<\frac{r}{C(b,\sigma,\mathcal{M})}  \\
	\quad \frac{r}{C(b,\sigma,\mathcal{M})}  \quad &\text{if}  \quad T\geq \frac{r}{C(b,\sigma,\mathcal{M})}\geq 1, \\
	\quad \left[ \frac{r}{C(b,\sigma,\mathcal{M})}\right]^{\sigma-1},&\text{if} \quad T\geq \frac{r}{C(b,\sigma,\mathcal{M})}, \quad 1\geq \frac{r}{C(b,\sigma,\mathcal{M})}   \quad\text{and} \quad 1<\sigma\leq 1.5 \\
	\quad \left[ \frac{r}{C(b,\sigma,\mathcal{M})}\right]^{2-\sigma}, &\text{if} \quad T\geq \frac{r}{C(b,\sigma,\mathcal{M})},\quad 1\geq \frac{r}{C(b,\sigma,\mathcal{M})}  \quad\text{and} \quad 1.5\leq<\sigma<2	
	\end{cases}
	\end{equation}
	and 
	\begin{equation}\label{CbsigmaM}
	C(b,\sigma,\mathcal{M})=
	\left[  \frac{\left| b\right| }{\Gamma(\sigma)}+\mathcal{M}\left(\frac{1+\Gamma(3-\sigma)}{2-\sigma}\right) \right]. 
	\end{equation}
	
\end{theorem}

\begin{proof}
	As it is known from Lemma \ref{lemma1}, solutions of problem \ref{intvalue} are solutions of integral equation \ref{intequ} as well. Moreover, the fixed points of the opeator $\mathcal{S}:C^{\sigma-1} [0,T_{0}]\to C^{\sigma-1} [0,T_{0}]$ defined by
	\begin{align} \label{soperator}
	\mathcal{S}\omega(x)=\frac{b}{\Gamma(\sigma)}x^{\sigma-1}+\int_{0}^{x}\frac{f(t,\omega(t),D^{\sigma-1}\omega(t))}{(x-t)^{1-\sigma }}dt 
	\end{align}
	interfere with solutions of the integral equation. For this reason, it is  sufficient to prove that operator $\mathcal{S}$ admits at least one fixed point. For this, it will be shown that operator $\mathcal{S}$ satisfies the hypotheses of Schauder fixed-point theorem. Let us start with showing the following inclusion to be valid:  
	$$\mathcal{S}(B_r)\subset B_r$$
	where 
	\begin{align*} 
	B_r=\{\omega \in C^{\sigma-1}[0,T_{0}]: ||\omega||_{\infty}+||D^{\sigma-1}\omega-b||_{\infty} \leq r \} 
	\end{align*}
	is a closed compact subset of $C^{\sigma-1}[0,T_{0}].$  Accordingly to norm on $C^{\sigma-1} [0,T_{0}],$  upper bounds of $\left\|\mathcal{S}\omega(x) \right\|_{\infty}$ and  $\left\| D^{\sigma-1}\mathcal{S}\omega(x)-b\right\|_{\infty}$ can be determined as follows: 
	
	\begin{align} 
	\left|\mathcal{S}\omega(x)\right| &\leq  \frac{\left|b\right| }{\Gamma(\sigma)}x^{\sigma-1}+\frac{1}{\Gamma(\sigma)}\int_{0}^{x}\frac{\left|t^{\sigma-1}f(t,\omega(t),D^{\sigma-1}\omega(t)\right|}{t^{\sigma-1}(x-t)^{1-\sigma }}dt \nonumber \\
	&\leq  \frac{\left|b\right| }{\Gamma(\sigma)}x^{\sigma-1}+\frac{\mathcal{M}}{\Gamma(\sigma)}\int_{0}^{1}\frac{x}{\tau^{\sigma-1}(1-\tau)^{1-\sigma }}d\tau 
    \leq  \frac{\left|b\right| }{\Gamma(\sigma)}x^{\sigma-1}+\Gamma(2-\sigma)\mathcal{M}x \label{soperator1}
	\end{align}
	and  
	\begin{align} 
	\left|D^{\sigma-1}\mathcal{S}\omega(x)-b\right| &\leq  \int_{0}^{x}\frac{\left|t^{\sigma-1}f(t,\omega(t),D^{\sigma-1}\omega(t)\right|}{t^{\sigma-1}}dt   \mathcal{M}x^{2-\sigma} \int_{0}^{1}\tau^{1-\sigma}d\tau= \frac{\mathcal{M}x^{2-\sigma}}{2-\sigma}. \label{soperator2}
	\end{align}
	From (\ref{soperator1}) and (\ref{soperator2}), 
	\begin{align}
	\left|\mathcal{S}\omega(x)\right|+\left|D^{\sigma-1}\mathcal{S}\omega(x)-b\right| \leq \frac{\left|b\right| }{\Gamma(\sigma)}x^{\sigma-1}+\Gamma(2-\sigma)\mathcal{M}x+ \frac{\mathcal{M}x^{2-\sigma}}{2-\sigma}  
	\end{align}
	is obtained. Taking supremum over $[0,T_0]$ for a $T_0 >0$ for the right hand-side of the above equation,   
	\begin{align}
	\left|\mathcal{S}\omega(x)\right|+\left|D^{\sigma-1}\mathcal{S}\omega(x)-b\right| \leq C(b,\sigma,\mathcal{M}) T_0 ^{\alpha}  
	\end{align}
	can be written, where $\alpha\in\Omega=\left\lbrace \sigma-1,1,2-\sigma\right\rbrace.$ $\alpha$ depends on values of $b,\mathcal{M},\sigma,r. $ To determine $T_0$ and  $\alpha$, let   $$C(b,\sigma,\mathcal{M}) T_0 ^{\alpha}=r.$$ 
	If $T_0 ^{\alpha}=\frac{r}{C(b,\sigma,\mathcal{M})}<1,$ then it is observed that $T_0 <1$ for any $\alpha\in \Omega.$ If $T_0 ^{\alpha} =\frac{r}{C(b,\sigma,\mathcal{M})}\geq1,$ it must be $T_0 \geq 1$
	for any $\alpha\in \Omega.$ Thus, 
	\begin{align}
	\sup_{x\in [0,T_0]}\left[ \left|\mathcal{S}\omega(x)\right|+\left|D^{\sigma-1}\mathcal{S}\omega(x)-b\right|\right]  \leq C(b,\sigma,\mathcal{M}) T_0 ^{\alpha}=r,  
	\end{align}
	where $$T_0:=\left[\frac{r}{C(b,\sigma,\mathcal{M})}\right]^{1/\alpha} $$
	and 
	\begin{equation}\label{alpha}
	\alpha= \begin{cases}
	\quad 1 \quad &\text{if} \quad  \frac{r}{C(b,\sigma,\mathcal{M})}\geq 1  \\
	\quad \sigma-1\quad &\text{if} \quad \frac{r}{C(b,\sigma,\mathcal{M})}<1 \quad \text{and} \quad  1<\sigma\leq1.5\\
	\quad 2-\sigma \quad &\text{if} \quad \frac{r}{C(b,\sigma,\mathcal{M})}<1 \quad \text{and} \quad  1.5\leq\sigma<2.\\
	\end{cases}
	\end{equation}
	
	Consequently, for all cases we obtain $$||\mathcal{S}\omega||_{\infty}+||D^{\sigma-1}\mathcal{S}\omega-b||_{\infty}\leq r,$$ which is the desired result.  
	
	Now, let us prove the equicontinuity of $\mathcal{S}(B_{r})\subset C^{\sigma-1} [0,T_{0}].$ Since the composition of uniformly continuous functions is so as well, the function $x^{\sigma-1}f(x,\omega(x),D^{\sigma-1}\omega(x
	))$ is uniformly continuous on $[0,T_{0}].$ Because for any  $\omega\in B_{r},$ both $\omega(x)$ and $D^{\sigma-1}\omega(x)$ and   $x^{\sigma-1}f(x,\omega,v)$ are
	uniformly continuous on $I,$ respectively. 
	Therefore, for given any $\epsilon>0,$ one can find a $%
	\delta=\delta(\epsilon)>0$ so that for all $x_{1},x_{2}\in [0,T_{0}]$ with $\left|x_{1}-x_{2}\right|<\delta$ it is
	$$\left \vert x_{1}^{\sigma-1}f(x_{1},\omega(x_{1}),D^{\sigma-1}\omega(x
	))-x_{2}^{\sigma-1}f(x_{2},\omega(x_{2}),D^{\sigma-1}\omega(x_{2}
	))\right \vert <K\epsilon ,$$
	where $K=\max\left(\frac{1}{T_{0}\Gamma(2-\sigma)}, \frac{2-\sigma}{T_{0}^{2-\sigma}}\right).$
	It follows that   
	\begin{align*}
     \big|\mathcal{S}&\omega\left(x_{1}\right)-\mathcal{S}\omega\left(x_{2}\right) \big|+\big|D^{\sigma-1}\mathcal{S}\omega\left(x_{1}\right)-D^{\sigma-1}\mathcal{S}\omega\left(x_{2}\right) \big| \\
     &\leq \int_{0}^{1} \frac{\left \vert
		h\left( \eta x_{1}\right) -h\left( \eta x_{2}\right)  \right
		\vert}{\Gamma \left(\sigma\right)\eta^{1-\sigma} \left( 1-\eta \right) ^{\sigma-1}}x d\eta+\int_{0}^{1} \frac{\left \vert
		h\left( \eta x_{1}\right) -h\left( \eta x_{2}\right)  \right
		\vert}{\eta^{1-\sigma} }x^{2-\sigma}d\eta\\
 &<T_{0}\Gamma(2-\sigma)K\epsilon+\frac{T_{0}^{2-\sigma}}{2-\sigma}K\epsilon=\epsilon,
	\end{align*}
	where $h(x)=x^{\sigma-1}f\left(x,\omega\left(x\right),D^{\sigma-1}\omega(x
	) \right).$ This implies that  $\mathcal{S}(B_{r})
	$ is an equicontinuous set of $C^\sigma [0,T_{0}].$ \\
	
	Finally, the continuity of $\mathcal{S}$ on $B_{r}$ will be proven. Assume  that $\left \{\omega_{k}\right \}_{k=1}^{\infty}\subset B_{r}$ is a sequence with $\omega_{k}\stackrel{C^\sigma [0,T_{0}]}{\rightarrow} \omega$ as $k\rightarrow \infty.$ Then, one can easily conclude that $\omega_{k}$ and $D^{\sigma-1}\omega_{k}(t)$ 
	converges uniformly to $\omega$ and $D^{\sigma-1}\omega(t
	),$ respectively. With these and the uniform continuity of $x^{\sigma-1}f(x,\omega,v)$ on $I=\left[0,T\right]\times\left[-r_{1},r_{1}\right] \times \left[ b-r_{2},b+r_{2}\right],$ it leads to
	\begin{align*}
	&\left \|\mathcal{S}\omega_{k}-\mathcal{S}\omega\right \|_{\sigma-1}=\sup_{x\in [0,T_{0}]} \left \vert \frac{1}{\Gamma \left(
		\sigma\right)}\int_{0}^{x}\frac{ \left[f\left(t,\omega_{k}(t
		),D^{\sigma-1}\omega_{k}(t
		)\right) -f\left(t ,\omega(t),D^{\sigma-1}\omega(t
		)\right)\right] }{\left(x-t\right)
		^{1-\sigma}}dt \right \vert\\
	&+\sup_{x\in [0,T_{0}]} \left \vert \int_{0}^{x} \left[f\left(t,\omega_{k}(t
	),D^{\sigma-1}\omega_{k}(t
	)\right) -f\left(t ,\omega(t),D^{\sigma-1}\omega(t
	)\right)\right] dt \right \vert
	\\
	&\leq \sup_{\eta x\in [0,T_{0}]} \int_{0}^{1}\frac{(\eta x)^{\sigma-1} \left|f\left(\eta x,\omega_{k}(\eta x),D^{\sigma-1}\omega_{k}(\eta x
		)\right)-f\left( \eta x ,\omega(\eta x),D^{\sigma-1}\omega(\eta x
		)\right)\right|}{\Gamma \left(\sigma\right)\eta^{\sigma-1}\left( 1-\eta \right) ^{1-\sigma}} xd\eta \\
	&+\sup_{\eta x\in [0,T_{0}]} \int_{0}^{1}\frac{(\eta x)^{\sigma-1} \left|f\left(\eta x,\omega_{k}(\eta x),D^{\sigma-1}\omega_{k}(\eta x
		)\right)-f\left( \eta x ,\omega(\eta x),D^{\sigma-1}\omega(\eta x
		)\right)\right|}{\eta^{\sigma-1}} x^{2-\sigma}d\eta \\
	&\rightarrow 0 \quad \text{as} \quad k\rightarrow \infty.
	\end{align*}  
	
	In conclusion, since hypotheses of Theorem \ref{Th2} are fulfilled, it implies that operator $\mathcal{S}$ admits at least one fixed point in $C^\sigma [0,T_{0}],$  which is a solution of problem (\ref{intvalue}) as well.
\end{proof}

The mean value theorem for R-L derivative of order $\sigma\in (0,1)$ was correctly given by \cite{San3}. Now, its counterparts for order $\sigma\in (1,2)$ reads as follows:
\begin{lemma}\label{lemma1.2}   
	Let $\sigma \in (1,2)$ and $\omega\in C^{\sigma-1}\left( \left[ 0,T\right]\right)  .$ Then, there is a function  $\mu:[0,T]\to [0,T]$ with $0<\mu(x)<x $ so that
	$$\omega(x)=-D^{\sigma-1}\omega(0)\frac{x^{\sigma-1}}{\Gamma(\sigma)}+\Gamma(2-\sigma)(\mu(x))^{\sigma-1}D^{\sigma-1}\omega(\mu(x)),$$
	is satisfied. 
\end{lemma}  

The lemma can be proved by following the way used in \cite{San3} and we omit it here. By the help of this lemma we can obtain the Nagumo type uniqueness:  

\begin{theorem} (\textit{Nagumo type uniqueness}) Let $1<\sigma<2,$ $0<T<\infty$ and  let condition {\rm \textbf{(C1)}} be satisfied. Moreover, assume that there exists a positive real number $L\leq \frac{2-\sigma}{T(1+\Gamma(3-\sigma))}$ such that the inequality
	\begin{align}\label{unique1}
	x^{\sigma-1}\left|f(x,t_{1,1},t_{2,1})-f(x,t_{1,2},t_{2,2})\right|\leq L\left( \left|t_{1,1}-t_{1,2}\right|+\left|t_{2,1}-t_{2,2}\right|\right)  
	\end{align}
	is fulfilled for all $x\in[0,T]$ and for all $t_{1,i},t_{2,i}\in\mathbb{R}$ with $i=1,2.$ Then, \eqref{intvalue} has at most one solution  in the space of $C^{\sigma-1}(\left[0,T_0\right]).$

\begin{proof}
	We have just showed the existence of the solution to problem \eqref{intvalue} in the previous theorem. For the uniqueness, we first assume that \eqref{intvalue} admits two different
	solutions such as $\omega_{1}$ and $\omega_{2}$ in the space of $C^{\sigma-1}(\left[0,T_0\right]).$ Let us define a function $\Phi(x)$ to be in the form
	$$
	\Phi(x):=\begin{cases}
	\left|\omega_{1}(x)-\omega_{2}(x)\right|+\left|D^{\sigma-1}\omega_{1}(x)-D^{\sigma-1}\omega_{2}(x)\right|,& x>0 \\
	\ \ \ \ \ \ \ \ \ \ \ \ \ \ \ \ \ \ \ \ 0 \ \ \ \ \ \ \  \ \ \ \  ,& x=0.
	\end{cases}
	$$ 
	
	Since $\omega_{1},\omega_{2}\in C^{\sigma-1}(\left[0,T\right]),$ the continuity of  $\Phi(x)$  on $x\in (0,T_0]$ can obviously be seen. For its continuity at $x=0,$  
	\begin{align*}
	0\leq\lim_{x\to 0^{+}}\Phi(x)&=\lim_{x\to 0^{+}} \frac{1}{\Gamma(\sigma)}\left|\int_{0}^{x}\frac{f\left(t ,\omega_{1}\left(t\right),D^{\sigma-1}\omega_{1}\left(t\right)\right)-f\left(t ,\omega_{2}\left(t\right),D^{\sigma-1}\omega_{2}\left(t\right)\right)}{\left(x-t\right) ^{1-\sigma}}
	dt\right| \\
	&+\lim_{x\to 0^{+}}\left|\int_{0}^{x}  f\left(t ,\omega_{1}\left(t\right),D^{\sigma-1}\omega_{1}\left(t\right)\right)-f\left(t ,\omega_{2}\left(t\right),D^{\sigma-1}\omega_{2}\left(t\right)\right)
	dt\right| \\
	&\leq\int_{0}^{1}\frac{\lim_{x\to 0^{+}}x\left| H\left(x\eta,\omega_{1}\left( x\eta\right)\right)-H\left(x\eta,\omega_{2}\left( x\eta\right)\right) \right| }{\eta^{\sigma-1}\left(1-\eta\right)^{1-\sigma}}d\eta \\
	&+\int_{0}^{1} \frac{\lim_{x\to 0^{+}}x^{2-\sigma} \left| H\left(x\eta,\omega_{1}\left( x\eta\right)\right)-H\left(x\eta,\omega_{2}\left( x\eta\right)\right) \right|}{\eta^{\sigma-1}}d\eta=0,	 
	\end{align*}
	where $H(x,\omega(x))=x^{\sigma-1}f\left(x,\omega\left(x\right),D^{\sigma-1}\omega(x
	) \right)$ and  we made the change of variable $t=x\eta$ and used condition  (C1), respectively. Consequently,  $\lim_{x\to 0^{+}}\Phi(x)=0=\Phi(0).$ \\ 
	
	The fact that  $\Phi(x)\geq 0$ on $[0,T]$ allows us to choose a point $x_{0}\in (0,T]$ so that
	\begin{align*}  
	0<\Phi(x_{0})&=\left|\omega_{1}(x_{0})-\omega_{2}(x_{0})\right| +\left|D^{\sigma-1}\omega_{1}(x_{0})-D^{\sigma-1}\omega_{2}(x_{0})\right|.
	\end{align*}
	By using the mean value theorem in Lemma \ref{lemma1.2}  
	\begin{align}\label{unique5}
	\left|\omega_{1}(x_{0})-\omega_{2}(x_{0})\right|&=\Gamma(2-\sigma)x_{0}\left|x_{0,1}^{\sigma-1}D^{\sigma}(\omega_{1}-\omega_{2})(x_{0,1})\right|+\left|D^{\sigma-1}\omega_{1}(x_{0})-D^{\sigma-1}\omega_{2}(x_{0})\right|\nonumber \\
	&=\Gamma(2-\sigma)x_{0}x_{0,1}^{\sigma-1}\left|f\left(x_{0,1},\omega_{1}\left(x_{0,1}\right)\right)-f\left(x_{0,1},\omega_{2}\left(x_{0,1}\right)\right)\right|
	\end{align}
	is  obtained for $x_{0,1}\in (0,x_{0}).$  
	
	Secondly, for the estimation of 
	$\left|D^{\sigma-1}\omega_{1}(x_{0})-D^{\sigma-1}\omega_{2}(x_{0})\right|,$ we have from the well-known integral mean theorem for the classical calculus
	\begin{align}\label{unique6}
	\left|D^{\sigma-1}\omega_{1}(x_{0})-D^{\sigma-1}\omega_{2}(x_{0})\right|&=
	\int_{0}^{x_0}\frac{t^{\sigma-1}\left|f(t,\omega_{1}(t),D^{\sigma-1}\omega_{1}(t)-f(t,\omega_{2}(t),D^{\sigma-1}\omega_{2}(t)\right|}{t^{\sigma-1}}dt \nonumber\\
	&=\frac{x^{2-\sigma}_{0}}{2-\sigma}x_{0,2}^{\sigma-1}\left|f\left(x_{0,2},\omega_{1}\left(x_{0,2}\right)\right)-f\left(x_{0,2},\omega_{2}\left(x_{0,2}\right)\right)\right|,
	\end{align}
	where $x_{0,2}\in (0,x_{0}).$   
	
	We specify $x_{1}$ as one of the points $x_{0,1}$ and $x_{0,2}$ so that   $\left| H(x_{1},\omega_{1}(x_{1}))-H(x_{1},\omega_{2}(x_{1}))\right|$ $ :=\max\left(\left| H(x_{0,1},\omega_{1}(x_{0,1}))-H(x_{0,1},\omega_{2}(x_{0,1}))\right|,\left| H(x_{0,2},\omega_{1}(x_{0,2}))-H(x_{0,2},\omega_{2}(x_{0,2}))\right| \right) . $ 
	
	Thus, from \eqref{unique5} and \eqref{unique6}, we have
	\begin{align} \label{unique7}
	0&<\Phi(x_{0})\leq\left(\Gamma(2-\sigma)x_{0}+\frac{x^{2-\sigma}_{0}}{2-\sigma}\right) \left| H(x_{1},\omega_{1}(x_{1}))-H(x_{1},\omega_{2}(x_{1}))\right|\nonumber\\
	&\leq T \left( \frac{1+\Gamma(3-\sigma)}{2-\sigma}\right)x_{1}^{\sigma-1} \left|f\left(x_{1},\omega_{1}\left(x_{1}\right)\right)-f\left(x_{1},\omega_{2}\left(x_{1}\right)\right)\right|\nonumber\\
	&\leq TL \left( \frac{1+\Gamma(3-\sigma)}{2-\sigma}\right)x_{1}^{\sigma-1} \left( \left|\omega_{1}(x_{1})-\omega_{2}(x_{1})\right| +\left|D^{\sigma-1}\omega_{1}(x_{1})-D^{\sigma-1}\omega_{2}(x_{1})\right|\right) =\Phi(x_{1})\nonumber
	\end{align}
	since $L\leq \frac{2-\sigma}{T(1+\Gamma(3-\sigma))}.$ Repeating the same procedure for the point $x_{1},$ it enables us to find some points $x_{2}\in(0,x_{1})$ so that
	$0<\Phi(x_{0})\leq\Phi(x_{1})\leq \Phi(x_{2}).$
	Continuing in the same way,  the sequence
	$\left\{x_{n}\right\}_{n=1}^{\infty}\subset [0,x_{0})$
	can be constructed so that  $x_{n}\to 0$ and
	\begin{equation}\label{uniq}
	0<\Phi(x_{0})\leq\Phi(x_{1})\leq\Phi(x_{2})\leq...\leq\Phi(x_{n})\leq... 
	\end{equation}
	
	However, the fact that $\Phi(x)$ is continuous at $x=0$ and
	$x_{n}\to 0$  leads to $\Phi(x_{n})\to \Phi(0)=0,$ and this
	contradicts with \eqref{uniq}. Consequently,  IVP
	\eqref{intvalue} possesses a unique solution.
\end{proof}
\end{theorem}

\begin{theorem} (\textit{Krasnoselskii-Krein type uniqueness}) Let $1<\sigma<2$  and $ T^{*}_{0}=\min\left\lbrace T_{0},1 \right\rbrace,$ where $T_0$ is defined by \eqref{T0}. Let condition (C1) be fulfilled. Furthermore, suppose that there exists a $L>0$  so that
	\begin{align}\label{unique2}
	x^{\sigma-1}\left|f(x,t_{1,1},t_{2,1})-f(x,t_{1,2},t_{2,2})\right|\leq \frac{L}{2}\left( \left|t_{1,1}-t_{1,2}\right|+\left|t_{2,1}-t_{2,2}\right|\right)  
	\end{align}
	holds for all $x\in[0,T]$ and for all $t_{1,i},t_{2,i}\in\mathbb{R}$ with $i=1,2,$  and that there exist $C>0$ and $\alpha \in (0,1)$ satisfying	 $(1-\sigma)(1-\alpha)-L(1-\alpha)+1>0$ such that
	\begin{align}\label{unique3}
	x^{\sigma-1}\left|f(x,t_{1,1},t_{2,1})-f(x,t_{1,2},t_{2,2})\right|\leq C\left( \left|t_{1,1}-t_{1,2}\right|+x^{\alpha(\sigma-1)}\left|t_{2,1}-t_{2,2}\right|\right) 	\end{align}
	holds for all $x\in[0,T]$ and for all $t_{1,i},t_{2,i}\in\mathbb{R}$ with $i=1,2.$ 
	Then, problem \eqref{intvalue} has an unique solution  in the space of $C^{\sigma-1}(\left[0,T^{*}_{0}\right]).$
 \begin{proof}
	As  claimed in the previous theorem, we first assume that problem \eqref{intvalue} has two different solutions such as $\omega_{1}(x)$ and $\omega_{2}(x)$ in $C^{\sigma-1}(\left[0,T^{*}_{0}\right]).$ However, by contradiction, we will show that it can not be happen. For this, let us first define $\Phi_{1}(x)=\left|\omega_{1}(x)-\omega_{2}(x)\right|$ and $\Phi_{2}(x)=\left|D^{\sigma-1}\omega_{1}(x)-D^{\sigma-1}\omega_{2}(x)\right|$ and try to find estimates for each functions by using condition {\rm \textbf{(C1)}} and  inequality \eqref{unique3}. Hence, we first have    
	\begin{align*}
	\Phi_{1}(x)&\leq \frac{1}{\Gamma(\sigma)}\int_{0}^{x}\frac{\left|f\left(t ,\omega_{1}\left(t\right),D^{\sigma-1}\omega_{1}\left(t\right)\right)-f\left(t ,\omega_{2}\left(t\right),D^{\sigma-1}\omega_{2}\left(t\right)\right)\right|}{\left(x-t\right) ^{1-\sigma}}
	dt \\
	&\leq \frac{1}{\Gamma(\sigma)}\int_{0}^{x}\frac{t^{\sigma-1}\left|f\left(t ,\omega_{1}\left(t\right),D^{\sigma-1}\omega_{1}\left(t\right)\right)-f\left(t ,\omega_{2}\left(t\right),D^{\sigma-1}\omega_{2}\left(t\right)\right)\right|}{t^{\sigma-1}\left(x-t\right) ^{1-\sigma}}
	dt \\
	&\leq \frac{C}{\Gamma(\sigma)}  \int_{0}^{x}\frac{\left[\Phi^{\alpha}_{1}(x)+t^{\alpha(\sigma-1)}\Phi^{\alpha}_{2}(x) \right] }{t^{\sigma-1}\left(x-t\right) ^{1-\sigma}}
	dt \\
	&\leq \frac{C}{\Gamma(\sigma)}\left(  \int_{0}^{x}\left(\frac{1}{t^{\sigma-1}\left(x-t\right) ^{1-\sigma}}\right) ^{q} 
	dt\right)^{1/q}\left(  \int_{0}^{x}\left[\Phi^{\alpha}_{1}(t)+t^{\alpha(\sigma-1)}\Phi^{\alpha}_{2}(t) \right]^{p} 
	dt\right)^{1/p} \\
	&\leq \frac{C\Gamma(1+(1-\sigma)q)\Gamma(1+(\sigma-1)q))}{\Gamma(\sigma)}x^{1/q}\Omega^{1/p}(x)
	\end{align*}
	where we used Hölder inequality with $q>1$ satisfying $(1-\sigma)q+1>0$ and $p=q/(q-1),$ and $\Omega(x)$ is defined by
	$$\Omega(x)=\int_{0}^{x}\left[\Phi^{\alpha}_{1}(t)+t^{\alpha(\sigma-1)}\Phi^{\alpha}_{2}(t) \right]^{p}dt.$$
	
	From here, we have the following estimation 
	\begin{align}\label{estimation1}
	\Phi^{p}_{1}(x)\leq Cx^{p/q}\Omega(x),
	\end{align}
	where $C$ is not specified here and throughout the proof. In addition to this, the upper bound for $\Phi_{2}(x)$ can be found as follows:
	\begin{align*}
	\Phi_{2}(x)&\leq \int_{0}^{x} \left|f\left(t ,\omega_{1}\left(t\right),D^{\sigma-1}\omega_{1}\left(t\right)\right)-f\left(t ,\omega_{2}\left(t\right),D^{\sigma-1}\omega_{2}\left(t\right)\right)\right| 
	dt \\
	&\leq \int_{0}^{x}\frac{t^{\sigma-1}\left|f\left(t ,\omega_{1}\left(t\right),D^{\sigma-1}\omega_{1}\left(t\right)\right)-f\left(t ,\omega_{2}\left(t\right),D^{\sigma-1}\omega_{2}\left(t\right)\right)\right|}{t^{\sigma-1} }
	dt \\
	&\leq \frac{C}{\Gamma(\sigma)}  \int_{0}^{x}\frac{\left[\Phi^{\alpha}_{1}(x)+t^{\alpha(\sigma-1)}\Phi^{\alpha}_{2}(x) \right] }{t^{\sigma-1}\left(x-t\right) ^{1-\sigma}}
	dt \\
	&\leq C \left(  \int_{0}^{x}\left(\frac{1}{t^{\sigma-1}}\right)^{q} 
	dt\right)^{1/q}\left(  \int_{0}^{x}\left[\Phi^{\alpha}_{1}(t)+t^{\alpha(\sigma-1)}\Phi^{\alpha}_{2}(t) \right]^{p} 
	dt\right)^{1/p} \\
	&\leq \frac{C\Gamma(1+(1-\sigma)q)}{\Gamma(2+(\sigma-1)q))}x^{(1+q(1-\sigma))/q}\Omega^{1/p}(x).
	\end{align*}
	From here, 
	\begin{align}\label{estimation2}
	\Phi^{p}_{2}(x)\leq Cx^{(1+q(1-\sigma))p/q}\Omega(x) 
	\end{align}
	is then obtained. By using estimations in \eqref{estimation1} and \eqref{estimation2} in the derivative of $\Omega(x)$ we have 
	\begin{align}\label{estimation3}
	\Omega'(x)&=\left[\Phi^{\alpha}_{1}(x)+t^{\alpha(\sigma-1)}\Phi^{\alpha}_{2}(x) \right]^{p}\leq 2^{p-1}\left[\left( \Phi^{\alpha}_{1}(x)\right) ^{p}+t^{p\alpha(\sigma-1)}\left( \Phi^{\alpha}_{2}(x)\right) ^{p}\right] \nonumber\\
	&= 2^{p-1}\left[\left( \Phi^{p}_{1}(x)\right)^{\alpha} +x^{p\alpha(\sigma-1)}\left( \Phi^{p}_{2}(x)\right)^{\alpha} \right]\nonumber\\
	&\leq 2^{p-1}\left[C^{\alpha}x^{\alpha p/q}\Omega^{\alpha}(x)+ C^{\alpha}x^{p\alpha(\sigma-1)}x^{(1+q(1-\sigma))\alpha p/q}\Omega^{\alpha}(x)\right]\leq Cx^{\alpha p/q}\Omega^{\alpha}(x). 
	\end{align}
	If we multiply the both sides of the above inequality with $(1-\alpha)\Omega^{-\alpha}(x),$ 
	\begin{align*}
	(1-\alpha)\Omega^{-\alpha}(x)\Omega'(x)=\frac{d}{dx}\left[\Omega^{1-\alpha}(x)\right] \leq Cx^{\alpha p/q}.
	\end{align*}
	is then obtained.  Integrating the both sides of the inequality over $[0,x],$ we get 
	\begin{align*}
	\Omega^{1-\alpha}(x)\leq Cx^{(\alpha p+q)/q},
	\end{align*}
	since $\Omega(0)=0.$ Consequently, this leads to the following estimation on $\Omega(x)$ 
	\begin{align}\label{estimation7}
	\Omega(x)\leq Cx^{(\alpha p+q)/(1-\alpha)q}.
	\end{align}
	By considering \eqref{estimation1} and  \eqref{estimation2} together with \eqref{estimation7}, one can conclude that
	\begin{align*}
	\Phi^{p}_{1}(x)\leq Cx^{p/q}\Omega(x)\leq Cx^{p/q}x^{(\alpha p+q)/(1-\alpha)pq}=Cx^{p+q/(1-\alpha)q}.
	\end{align*}
	or 
	\begin{align} \label{est4}
	\Phi_{1}(x)\leq Cx^{(p+q)/(1-\alpha)pq}=x^{1/(1-\alpha)},
	\end{align}
	and 
	\begin{align*}
	\Phi^{p}_{2}(x)\leq Cx^{p(1+q(1-\sigma))/q}\Omega(x)\leq Cx^{p(1+q(1-\sigma))/q}x^{(\alpha p+q)/(1-\alpha)pq}=Cx^{\frac{(1-\alpha)(1-\sigma)pq+p+q}{(1-\alpha)q}}.
	\end{align*}
	or
	\begin{align} \label{est5}
	\Phi_{2}(x)\leq Cx^{(1-\sigma)+\frac{p+q}{(1-\alpha)pq}}=x^{(1-\sigma)+\frac{1}{(1-\alpha)}},
	\end{align}
	since $\frac{p+q}{pq}=1.$ \\
	
	Let us now define $$\Psi(x)=x^{-L}\max\left\lbrace \Phi_{1}(x),\Phi_{2}(x)\right\rbrace,$$ 
	where $L(1-\alpha)<1+(1-\sigma)(1-\alpha).$ 
	If $\Phi_{1}(x)=\max\left\lbrace \Phi_{1}(x),\Phi_{2}(x)\right\rbrace,$ then from \eqref{est4} we get 
	$$0\leq \Psi(x)\leq x^{\frac{1}{1-\alpha}-L},$$  
	or in the case of $\Phi_{2}(x)=\max\left\lbrace \Phi_{1}(x),\Phi_{2}(x)\right\rbrace,$ by the inequality \eqref{est5} we have the following 
	$$0\leq \Psi(x)\leq x^{(1-\sigma)+\frac{1}{1-\alpha}-L}=x^{\frac{(1-\sigma)(1-\alpha)-L(1-\alpha)+1}{1-\alpha}}.$$  
	
	In both cases, $\Psi(x)$ is continuous on $[0,T^{*}_{0}]$ and  $\Psi(0)=0.$  Let us now show that $\Psi(x)\equiv 0$ on $[0,T^{*}_{0}].$ For this, suppose otherwise and let $\Psi(x)\not\equiv 0.$ This means $\Psi(x)>0,$ and from its continuity one can say that there exists a point $x_{1}\in [0,T^{*}_{0}]$ so that $\Psi(x)$ takes its maximum value at that point. Thus, let $$M=\Psi(x_{1})=\max_{x\in [0,T^{*}_{0}]}\Psi(x).$$ 
	
	By assuming $\Psi(x)=x^{-L}\Phi_{1}(x),$  
	\begin{align}
	M&=\Psi(x_{1})=x_{1}^{-L}\Phi_{1}(x_{1}) \\
	&\leq\frac{x_{1}^{-L}}{\Gamma(\sigma)}\int_{0}^{x_{1}}\frac{t^{\sigma-1}\left|f\left(t ,\omega_{1}\left(t\right),D^{\sigma-1}\omega_{1}\left(t\right)\right)-f\left(t ,\omega_{2}\left(t\right),D^{\sigma-1}\omega_{2}\left(t\right)\right)\right|}{t^{\sigma-1}\left(x-t\right) ^{1-\sigma}}dt \nonumber
	\\
	&\leq\frac{Lx_{1}^{-L}}{2\Gamma(\sigma)}\int_{0}^{x_{1}}\left(x-t\right) ^{\sigma-1}t^{1-\sigma+k}\left[\Phi_{1}(t)+\Phi_{2}(t) \right] dt \nonumber\leq\frac{Lx_{1}^{-L}}{\Gamma(\sigma)}\int_{0}^{x_{1}}\left(x-t\right) ^{\sigma-1}t^{1-\sigma+k}\Psi(t) dt \nonumber \\
	&\leq M\frac{L\Gamma(2-\sigma+L)}{\Gamma(2+L)}x_{1}< M  \nonumber
	\end{align}
	is obtained for  $t\in [0,T^{*}_{0}],$  since $\frac{L\Gamma(2-\sigma+L)}{\Gamma(2+L)}<1$ for $1<\sigma<2.$ However, it is a contradiction. 
	
	On the other hand, when  $\Psi(x)=x^{-L}\Phi_{2}(x),$ we get 
	\begin{align}
	M&=\Psi(x_{1})=x_{1}^{-L}\Phi_{2}(x_{1})\\
	&\leq x_{1}^{-L}  \int_{0}^{x_{1}}\frac{t^{\sigma-1}\left|f\left(t ,\omega_{1}\left(t\right),D^{\sigma-1}\omega_{1}\left(t\right)\right)-f\left(t ,\omega_{2}\left(t\right),D^{\sigma-1}\omega_{2}\left(t\right)\right)\right|}{t^{\sigma-1}}dt \nonumber
	\\
	&\leq\frac{Lx_{1}^{-L}}{2}\int_{0}^{x_{1}}t^{1-\sigma+L}\left[\Phi_{1}(t)+\Phi_{2}(t) \right] dt \leq L x_{1}^{-L} \int_{0}^{x}t^{1-\sigma+L}\Psi(t) dt \nonumber \\
	&\leq M\frac{L}{L+2-\sigma}x^{L+2-\sigma}_{1}< M,  \nonumber
	\end{align}
	which is contraction as well. 
	
	Consequently, it must be $\Psi(x)\equiv 0$ on $[0,T^{*}_{0}].$ This gives us the uniqueness of solutions of the considered problem. 
\end{proof}
\end{theorem}

\begin{theorem} [Osgood-type uniqueness] Let $1<\sigma<2,$  and let $T_0$ be defined by \eqref{T0} and condition (C1) be satisfied. Furthermore, suppose that for all $x\in[0,T]$ and for all $t_{1,i},t_{2,i}\in\mathbb{R}$ with $i=1,2,$  the equality
\begin{align}\label{osgood}
x^{\sigma-1}\left|f(x,t_{1,1},t_{2,1})-f(x,t_{1,2},t_{2,2})\right|\leq C \left( g\left(\left|t_{1,1}-t_{1,2}\right|^{p}+\left|t_{2,1}-t_{2,2}\right|^{p}\right)\right) ^{1/p}  
\end{align}
is fulfilled, where $p>1$ is conjugate of $q>1$ satisfying  $1+(1-\sigma)q>0,$
$$C^{q}\geq 2\max\left(  b\Gamma^{q}(\sigma)[T_0\Gamma(1+(1-\sigma)q)\Gamma(1+(\sigma-1)q)]^{-1} ,  (1+(1-\sigma)q)  T_{0}^{-1-q(1-\sigma)} \right) ,$$
and, $g$ is a continuous, non-negative and  non-decreasing function in  $[0,T_0]$ so that $g(0)=0$  and it satisfies
	\begin{align}\label{osgood3}
	\lim_{\epsilon\to 0^{+}}\int_{\epsilon}^{\gamma}\frac{du}{g(u)}=\infty
	\end{align}
	for any $\gamma\in\mathbb{R}.$ 	
	Then, \eqref{intvalue} has an unique solution  in the space of $C^{\sigma-1}(\left[0,T^{*}_{0}\right]).$

\begin{proof}	As made in previously given uniqueness theorems, we assume that there exist two different solutions such as $\omega_{1}(x)$ and $\omega_{2}(x)$ to problem \eqref{intvalue} in $C^{\sigma-1}(\left[0,T^{*}_{0}\right]).$ Moreover, let 
	$\Phi_{1}(x)=\left| \omega_{1}(x)-\omega_{2}(x)\right| $  and $\Phi_{2}(x)=\left| D^{\sigma-1}\omega_{1}(x)-D^{\sigma-1}\omega_{2}(x)\right| .$ At first, we get the estimation on $\Phi_{1}(x)$ as follows:
	\begin{align}
	\Phi_{1}(x)&\leq \frac{1}{\Gamma(\sigma)}\int_{0}^{x}\frac{t^{\sigma-1}\left|f\left(t ,\omega_{1}\left(t\right),D^{\sigma-1}\omega_{1}\left(t\right)\right)-f\left(t ,\omega_{2}\left(t\right),D^{\sigma-1}\omega_{2}\left(t\right)\right)\right|}{t^{\sigma-1}\left(x-t\right) ^{1-\sigma}}
	dt \nonumber\\
	&\leq \frac{C}{\Gamma(\sigma)}  \int_{0}^{x}\frac{\left[  g\left(\left|\omega_{1}\left(t\right)-\omega_{2}\left(t\right)\right|^{p}+\left|D^{\sigma-1}\omega_{1}\left(t\right)-D^{\sigma-1}\omega_{2}\left(t\right)\right|^{p}\right)\right]  ^{1/p}}{t^{\sigma-1}\left(x-t\right) ^{1-\sigma}}
	dt \nonumber\\
	&\leq \frac{C}{\Gamma(\sigma)}\left(  \int_{0}^{x}\left(\frac{1}{t^{\sigma-1}\left(x-t\right) ^{1-\sigma}}\right) ^{q} 
	dt\right)^{1/q}\left(  \int_{0}^{x}  g\left(\Phi^{p}_{1}(t)+\Phi^{p}_{2}(t)\right) 
	dt\right)^{1/p} \nonumber\\
	&\leq C\left[ \frac{\Gamma(1+(1-\sigma)q)\Gamma(1+(\sigma-1)q)}{\Gamma^{q}(\sigma)}\right]^{1/q}x^{1/q}\left( \int_{0}^{x}  g\left(\Phi^{p}_{1}(t)+\Phi^{p}_{2}(t)\right) 
	dt\right)^{1/p} \nonumber\\
	&\leq \frac{1}{2^{1/p}}\left( \int_{0}^{x} g\left(\Phi^{p}_{1}(t)+\Phi^{p}_{2}(t)\right) 
	dt\right)^{1/p}
	\end{align}	
	where we used the inequality \eqref{osgood}, Hölder inequality and the assumption on $C,$ respectively. From here, it follows that 
	\begin{align}\label{estim1}
	\Phi^{p}_{1}(x)\leq  \frac{1}{2}\int_{0}^{x} g\left(\Phi^{p}_{1}(t)+\Phi^{p}_{2}(t)\right) 
	dt .
	\end{align}
	Similarly to above, we have
	\begin{align}
	\Phi_{2}(x)&\leq \int_{0}^{x}\frac{t^{\sigma-1}\left|f\left(t ,\omega_{1}\left(t\right),D^{\sigma-1}\omega_{1}\left(t\right)\right)-f\left(t ,\omega_{2}\left(t\right),D^{\sigma-1}\omega_{2}\left(t\right)\right)\right|}{t^{\sigma-1}}
	dt \nonumber\\
	&\leq C  \int_{0}^{x}\frac{\left[  g\left(\left|\omega_{1}\left(t\right)-\omega_{2}\left(t\right)\right|^{p}+\left|D^{\sigma-1}\omega_{1}\left(t\right)-D^{\sigma-1}\omega_{2}\left(t\right)\right|^{p}\right)\right]  ^{1/p}}{t^{\sigma-1}}
	dt \nonumber\\
	&\leq C\left(  \int_{0}^{x}\left(\frac{1}{t^{\sigma-1}}\right) ^{q} 
	dt\right)^{1/q}\left(  \int_{0}^{x}  g\left(\Phi^{p}_{1}(t)+\Phi^{p}_{2}(t)\right) 
	dt\right)^{1/p} \nonumber\\
	&\leq C\left[ \frac{1}{(1+(1-\sigma)q)}\right]^{1/q} x^{(1+q(1-\sigma))/q}\left( \int_{0}^{x}  g\left(\Phi^{p}_{1}(t)+\Phi^{p}_{2}(t)\right) 
	dt\right)^{1/p} \nonumber\\
	&\leq \frac{1}{2^{1/p}}\left( \int_{0}^{x} g\left(\Phi^{p}_{1}(t)+\Phi^{p}_{2}(t)\right) 
	dt\right)^{1/p}.
	\end{align}	
	This leads to 
	\begin{align} \label{estim2}
	\Phi^{p}_{2}(x)\leq  \frac{1}{2}\int_{0}^{x} g\left(\Phi^{p}_{1}(t)+\Phi^{p}_{2}(t)\right) 
	dt .
	\end{align}
	
	Now, set 
	$$ \Psi(x):=\max_{0\leq t\leq x} \left[ \Phi^{p}_{1}(x)+\Phi^{p}_{2}(x)\right],$$
	and assume that  $\Psi(x)>0$ for $x\in (0,T_{0}].$ We will show that it can not be possible under assumptions. 
	
	From the definition of $ \Psi(x)$ one easily conclude that for each $x\in [0,T_0],$
	it is  $\Phi^{p}_{1}(x)+\Phi^{p}_{2}(x)\leq  \Psi(x)$ and there exists a $x_{1}\leq x$ so that $ \Psi(x)=\Phi^{p}_{1}(x_{1})+\Phi^{p}_{2}(x_{1}).$ Then, from estimations \eqref{estim1}-\eqref{estim2} and from the fact that $g$ is non-decreasing function
	\begin{align}
	\Psi(x)=\Phi^{p}_{1}(x_{1})+\Phi^{p}_{2}(x_{1})\leq  \int_{0}^{x_{1}} g\left(\Phi^{p}_{1}(t)+\Phi^{p}_{2}(t)\right)\leq \int_{0}^{x} g\left(\Psi(t)\right)dt:=\Psi_{*}(x)
	\end{align}	
	is then obtained.  It can be seen that $\Psi(x)\leq \Psi_{*}(x).$ Moreover, we have 
	$$\frac{d}{dx}\Psi_{*}(x)=g\left(\Psi(x)\right)\leq g\left(\Psi_{*}(x)\right)$$   
	for all $x\in [0,T_0].$ From this fact, for sufficiently small $\delta>0,$ we have 
	$$\int_{\delta}^{T_{0}}\frac{\Psi^{'}_{*}(x)}{g\left(\Psi_{*}(x)\right)}dx\leq T_{0}-\delta.$$
	Furthermore, by changing variables $u=\Psi_{*}(x)$ in the above integral and by using the continuity of $\Psi_{*}(x)$ and $\Psi_{*}(0)=0$ , we have 
	$$\int_{\epsilon}^{\gamma}\frac{du}{g\left(u\right)} \leq T_{0}-\delta.$$
	for sufficiently small $\epsilon>0$ with $\epsilon=\Psi_{*}(\delta)$ and for $\gamma=\Psi_{*}(T_{0}).$ However, this contradicts with the assumption on $g$ given in \eqref{osgood3}. Consequently, $\Psi(x)=0$ for $x\in [0,T_{0}],$ i.e. $\omega_{1}(x)=\omega_{2}(x).$
\end{proof}
\end{theorem}

\textit{Remark.} It must be pointed out that, as noted in Theorem 1.4.3 in \cite{Agarwal}, the condition that function $g(u)$ is non-decreasing can be dropped.

\section{Conclusions }
In this research, we gave some sufficient conditions for existence and uniqueness of a problem involving a nonlinear differential equations in the sense of R-L derivative when the right-hand side function has a discontinuity at zero. Considering the literature, these results can be generalized and improved. Besides, one can obtain another uniqueness results for this problem as well.

\bigskip

\end{document}